\documentclass[12pt,leqno]{amsart}
\usepackage{amssymb,amsthm,amsmath,latexsym}
\newcommand{\vfi}{\varphi}

\newtheorem{theorem}{\sc Theorem}[section]

\newtheorem{lem}[theorem]{\sc Lemma}
\newtheorem{prop}[theorem]{\sc Proposition}
\newtheorem{cor}[theorem]{\sc Corollary}

\newtheorem{rem}[theorem]{\sc Remark}

\newtheorem*{thmA}{Theorem A}
\newtheorem*{thmB}{Theorem B}
\newtheorem*{thmC}{Theorem C}

\title[Non-abelian tensor product]{Finiteness conditions for the non-abelian tensor product of groups}
\author[Bastos]{R. Bastos}
\address{ Departamento de Matem\'atica, Universidade de Bras\'ilia,
Brasilia-DF, 70910-900 Brazil }
\email{bastos@mat.unb.br}
\author[Nakaoka]{I.\,N. Nakaoka}
\address{ Departamento de Matem\'atica, Universidade Estadual de Maring\'a, Maring\'a-PR, 87020-900 Brazil }
\email{innakaoka@uem.br}
\author[Rocco]{N.\,R. Rocco }
\address{ Departamento de Matem\'atica, Universidade de Bras\'ilia,
Brasilia-DF, 70910-900 Brazil }
\email{norai@unb.br}

\begin{document}

\maketitle

\begin{abstract}
Let $G$, $H$ be groups. We denote by $\eta(G,H)$ a certain extension of the non-abelian tensor product $G \otimes H$ by $G \times H$. We prove that if $G$ and $H$ are groups that act compatibly on each other and  such that the set of all tensors $T_{\otimes}(G,H)=\{g\otimes h \, : \, g \in G, \, h\in H\}$ is finite, then the non-abelian tensor product $G \otimes H$ is finite. In the opposite direction we examine certain finiteness conditions of $G$ in terms of similar conditions for the tensor square $G \otimes G$.
\end{abstract}

\maketitle

\section{Introduction}

Let $G$ and $H$ be groups each of which acts upon the other (on the right),
\[
G\times H \rightarrow G, \; (g,h) \mapsto g^h; \; \; H\times G \rightarrow
H, \; (h,g) \mapsto h^g
\]
and on itself by conjugation, in such a way that for all $g,g_1 \in G$ and
$h,h_1 \in H$,
\begin{equation}   \label{eq:0}
g^{\left( h^{g_1} \right) } = \left( \left( g^{g^{-1}_1}  \right) ^h \right) ^{g_1} \; \; \mbox{and} \; \; h^{\left( g^{h_1}\right) } =
\left( \left( h^{h_1^{-1}} \right) ^g \right) ^{h_1}.
\end{equation}
In this situation we say that $G$ and $H$ act {\em compatibly} on each other. Let  $H^{\varphi}$ be
an extra copy of $H$, isomorphic via $\varphi : H \rightarrow
H^{\varphi}, \; h \mapsto h^{\varphi}$, for all $h\in H$.
Consider the group $\eta(G,H)$ defined in  \cite{Nak} as
$$\begin{array}{ll} {\eta}(G,H) =  \langle
G,H^{\varphi}\ |  &
[g,{h}^{\varphi}]^{g_1}=[{g}^{g_1},({h}^{g_1})^{\varphi}], \;
[g,{h}^{\varphi}]^{h^{\varphi}_1} = [{g}^{h_1},
({h}^{h_1})^{\varphi}] , \\ & \ \forall g,g_1 \in G, \; h, h_1 \in H
\rangle . \end{array}$$
We observe that when $G=H$ and all actions are conjugations, $\eta (G,H)$ becomes the group $\nu (G)$ introduced in \cite{NR1}.

It is a well known fact (see \cite[Proposition 2.2]{Nak}) that the subgroup
$[G, H^{\varphi}]$ of $\eta(G,H)$ is canonically isomorphic with the {\em non-abelian
tensor product} $G \otimes H$, as defined by R. Brown and J.-L. Loday in their seminal paper \cite{BL}, the isomorphism being induced by $g \otimes h \mapsto
[g, h^{\varphi}]$ (see also Ellis and Leonard \cite{EL}). It is clear that the subgroup $[G,H^{\varphi}]$ is normal in $\eta(G,H)$ and one has the decomposition
\begin{equation} \label{eq:decomposition}
 \eta(G,H) = \left ( [G, H^{\varphi}] \cdot G \right ) \cdot H^{\varphi},
\end{equation}
where the dots mean (internal) semidirect products. An element $\alpha \in \eta(G,H)$ is called a {\em tensor} if $\alpha = [a,b^{\varphi}]$ for suitable $a\in G$ and $b\in H$. If $N$ and $K$ are subgroups of $G$ and $H$, respectively, let $T_{\otimes}(N,K)$  denote the set of all tensors $[a,b^{\varphi}]$ with $a\in N$ and $b\in K$. In particular, $[N,K^{\varphi}] = \langle  T_{\otimes}(N,K)  \rangle$. When $G=H$ and all actions are by conjugation, we simply write $T_{\otimes}(G)$ instead of $T_{\otimes}(G,G)$. Moreover, $[G,G^{\varphi}]$ denotes the non-abelian tensor square $G \otimes G$. We denote by $\Delta(G)$ the diagonal subgroup of $[G,G^{\varphi}]$, $\Delta(G)= \langle [g,g^{\varphi}] \mid \ g \in G \rangle$.

In the present paper we want to study the following question: If we assume certain restrictions on the set $T_{\otimes}(G,H)$, how does this influence in the structure of the groups $G \otimes H$ or $\eta(G,H)$?

In \cite{Ros} Rosenlicht proved that if $N$ and $K$ are subgroups of a group $M$, with $N$ normal in $M$, and if the set of commutators $\{[n,k] \, :\, n\in N, \; k\in K\}$ is finite, then so is the commutator subgroup $[N,K]$. Under appropriate conditions we can extend this result to the subgroup $[N,K^{\varphi}]$ of $\eta(G,H)$.

\begin{thmA}
Let $G$ and $H$ be groups that act compatibly on each other and suppose that $N$ and $K$ are subgroups of $G$ and $H$, respectively, such that $N$ is $K$-invariant and $K$ is $N$-invariant. If the set $T_{\otimes}(N,K)$ is finite, then so is the subgroup $[N,K^{\varphi}]$ of $\eta (G,H)$. In particular, if $T_{\otimes}(G,H)$ is finite, then so is the non-abelian tensor product $G \otimes H$.
\end{thmA}

An immediate consequence of the above theorem is a well-known result due to Ellis \cite{Ellis} (see also \cite{T}), that $G \otimes H$ is finite when $G$ and $H$ are finite groups. In the opposite direction one could be interested in studying conditions under which the finiteness of the $G \otimes H$ implies that of $G$ and $H$; in general, the finiteness of $G \otimes H$ does not implies the finiteness of the groups involved. An easy counter-example is provided just by taking $G= C_2 \times C_{\infty}$, $H=C_2$ and supposing that all actions are trivial; then the non-abelian tensor product $G \otimes H \cong G^{ab} \otimes_{\mathbb{Z}} H^{ab}$ is finite (see \cite{BL} for details), but $G$ contains elements of infinite order. The question is more interesting when $G=H$ and all actions are conjugations, although it is also well-known that, in general, the finiteness of the non-abelian tensor square $[G,G^{\varphi}]$ does not imply that $G$ is a finite group (see Remark \ref{prufer}, below). However, in \cite{NP} Parvizi and Niroomand proved that if $G$ is a finitely generated group in which the non-abelian tensor square is finite, then $G$ is finite.

In the sequel we consider certain finiteness conditions for the group $G$ in terms of the torsion elements of the non-abelian tensor square $[G,G^{\varphi}]$. We establish the following related result, which is also related to one due to Moravec \cite{M} who proved that if $G$ is a locally finite group, then so are the groups $[G,G^{\varphi}]$ and $\nu(G)$.

\begin{thmB} \label{thm.fg}
Let $G$ be a group with finitely generated abelianization. The following properties are equivalent:
\begin{itemize}
\item[$(a)$] $G$ is locally finite;
\item[$(b)$] The non-abelian tensor square $[G,G^{\varphi}]$ is locally finite;
\item[$(c)$] The derived subgroup $G'$ is locally finite and the diagonal subgroup $\Delta(G)$ is periodic.
\end{itemize}
\end{thmB}

In the same paper \cite{NP} Parvizi and Niroomand showed that if $G$ is a group with finitely generated abelianization and the non-abelian tensor square $[G,G^{\varphi}]$ is a $p$-group, then $G$ is a $p$-group. We extend this result to $\pi$-groups, where $\pi$ is a set of primes.

\begin{thmC} \label{thm.pi}
Let $\pi$ be a set of primes and $G$ a group with finitely generated abelianization. Suppose that the non-abelian tensor square $[G,G^{\varphi}]$ is a $\pi$-group. Then $G$ is a $\pi$-group.
\end{thmC}

In view of Theorem C one might suspect that similar phenomenon holds for an arbitrary non-abelian tensor product $G \otimes H$. However, the same counter-example given before, by taking $G= C_2 \times C_{\infty}$, $H=C_2$ and supposing that all actions are trivial, shows that $G \otimes H \cong G^{ab} \otimes_{\mathbb{Z}} H^{ab}$ is finite, but $G$ contains elements of infinite order. \\

In Section 4 we obtain some local finiteness criteria related to (locally) residually finite  groups $G$ and their respective non-abelian tensor squares $[G,G^{\varphi}]$, in terms of the set of tensors $T_{\otimes}(G)$.

\section{Preliminary results}

Note that there is an epimorphism $\rho: \nu(G) \to G$, given by $g \mapsto g$, $h^{\vfi} \mapsto h$, which induces the derived map $\rho':[G,G^{\vfi}] \to G'$, $[g,h^{\vfi}] \mapsto [g,h]$, for all $g,h \in G$. In the notation of \cite[Section 2]{NR2}, let $\mu(G)$ denote the kernel of $\rho'$, a central subgroup of $\nu(G)$. In particular, $$\dfrac{[G,G^{\vfi}]}{\mu(G)} \cong G'.$$

The next lemma is a particular case of \cite[Theorem 3.3]{NR1}.

\begin{lem} \label{thm.Rocco}
Let $G$ be a group. Then the derived subgroup $$\nu(G)' = ([G,G^{\vfi}] \cdot G') \cdot (G')^{\vfi},$$ where ``$\cdot$'' denotes an internal semi-direct product.
\end{lem}

\begin{lem} \label{lem.central}
Let $X$ be a normal subset of a group $G$. If $X$ is finite, then the subgroup $\langle X\rangle$ is central-by-finite.
\end{lem}

\begin{proof}
For every element $x \in X$, the conjugacy class $x^G$ has at most $\vert X \vert$ elements. It follows that the index $$[G:C_G(x)] \leqslant \vert X \vert,$$ for every $x \in X$. According to Poincar\'e's Lemma \cite[1.3.12]{Rob}, the index $[G: \cap_{x \in X} C_G(x)]$ is also finite. In particular, the subgroup  $\langle X \rangle$ is central-by-finite, since the subgroup $  \left( \cap_{x \in X} C_G(x) \right) \cap \langle X \rangle$ is central and has finite index in $\langle X \rangle$. The proof is complete.
\end{proof}

The following result is a consequence of \cite[Proposition 2.3]{BL}.

\begin{lem}  \label{ident}  Let $G$ and $H$  be groups acting compatibly on one other.  The following relations hold in $\eta(G,H)$, for all $g,x\in G$ and $h,y\in H$:

\begin{enumerate}
\item[(a)]  $[g,h]^{[x,y^{\varphi }]}=[g,h^{\varphi
}]^{x^{-1}x^{y}}=[g,h^{\varphi }]^{(y^{-x}y)^{\varphi }};$

\item[(b)] $[g^{-1}g^h,y^{\varphi }]=[g,h^{\varphi }]^{-1}[g,h^{\varphi }]^{y^{\varphi }}.$
\end{enumerate}
\end{lem}

\section{Proofs of the main results}

\begin{thmA}
Let $G$ and $H$ be groups that act compatibly on each other and suppose that $N$ and $K$ are subgroups of $G$ and $H$, respectively, such that $N$ is $K$-invariant and $K$ is $N$-invariant. If the set $T_{\otimes}(N,K)$ is finite, then so is the subgroup $[N,K^{\varphi}]$ of $\eta (G,H)$. In particular, if $T_{\otimes}(G,H)$ is finite, then so is the non-abelian tensor product $G \otimes H$.
\end{thmA}

\begin{proof} Since $N$ is $K$-invariant and $K$  is  $N$-invariant, from the defining relations of $\eta(G,H)$ it follows that $T_{\otimes}(N,K)$  is a normal subset of $\langle N,K^{\varphi}\rangle$. Thus,  the subgroup $[N,K^{\varphi}]$ is normal in $\langle N,K^{\varphi}\rangle$ and Lemma \ref{lem.central} gives us that it is also central-by-finite. Then, by Schur's Theorem \cite[14.1.4]{Rob}, the subgroup $[N,K^{\varphi}]^{\prime}$ is finite. Without loss of generality we may assume that $[N,K^{\varphi}]$ is abelian.

We claim that the subgroup $S=[N,K^{\varphi},K^{\varphi}]$ is normal in $\langle N,K^{\varphi} \rangle$ and finite. In fact, let $m\in [N,K^{\varphi}]$, $k,h \in K$ and $n\in N$. As $m$ and $[k^{-\varphi},n^{-1}]$ commute, we obtain
\[ \begin{array}{ccl}
[m,k^{\varphi}]^n &=& n^{-1}m^{-1}k^{-\varphi} m[k^{-\varphi},n^{-1}]nk^{\varphi} \\ &=&
n^{-1}m^{-1}k^{-\varphi} [k^{-\varphi},n^{-1}]mnk^{\varphi}\\
&=&[m^n,k^{\varphi}]
\end{array}
\]
and $[m,k^{\varphi}]^{h^{\varphi}}=[m^{h^{\varphi}},(k^h)^\varphi]$. Now, the normality of $S$ follows from the fact that $N$ and $K^{\varphi}$ normalize $[N,K^{\varphi}]$. We observe that the abelian group $S$ is generated by the set $X=\{[n,k^{\varphi}, h^{\varphi}]\, : \, n\in N, k,h\in K\}$, which is finite because $[n,k^{\varphi}, h^{\varphi}]=[n,k^{\varphi}]^{-1}[n^h,(k^h)^{\varphi}]
\in T^{-1}T$, for  all $n\in N,$ $ k,h\in K$, where $T=T_{\otimes}(N,K)$ and $T^{-1}=\{t^{-1}\,:\, t\in T\}$. Further, given $n\in N$ and $h,k\in K$, if $m=[n,h^{\varphi}]$ and $n_1=n^{-1}n^h$, then  $[m,k^{\varphi}]=[n,h^{\varphi},k^{\varphi}]=[n_1,k^{\varphi}]$ (by Lemma \ref{ident} (b)). Using these equalities, we get
\[
\begin{array}{ccl} [n,h^{\varphi},k^{\varphi}]^2 &=& [n_1,k^{\varphi}][m,k^{\varphi}] \\ &=&
m^{-1}[n_1,k^{\varphi}]k^{-\varphi}mk^{\varphi} \\&=&
m^{-1}[m,k^{\varphi}]k^{-\varphi}mk^{\varphi}\\ &=& [m^2,k^{\varphi}] \\ &=&
[m,k^{\varphi}]^m[m,k^{\varphi}]\\&=&
[n_1,k^{\varphi}]^m[n_1,k^{\varphi}]\\ &=&[{n_1}^2,k^{\varphi}] \qquad \mbox{(by Lemma \ref{ident} (a))}
\end{array}
\]
that is, $ [n,h^{\varphi},k^{\varphi}]^2 \in T$. We conclude that $S$ is a finitely generated abelian torsion group and, consequently, it is finite. Hence, we may assume that $K^{\varphi}$
centralizes $[N,K^{\varphi}]$. Since
\[[n,k^{\varphi}]^2= n^{-1}k^{-\varphi}nk^{\varphi}[n,k^{\varphi}]=
n^{-1}k^{-\varphi}n[n,k^{\varphi}]k^{\varphi}=[n,(k^2)^{\varphi}]\in T,\]
for all $n \in N$, $k\in K$, we obtain that  $[N,K^{\varphi}]$ is finite. The proof is complete.
\end{proof}

The remaining of this section will be devoted to obtain finiteness conditions for a group $G$ in terms of the orders of the tensors. Our proof involves looking at the description of the diagonal subgroup $\Delta(G) \leqslant [G,G^{\varphi}]$, where $G^{ab}$ is finitely generated. Such a description has previously been used by the third named author \cite{NR2}. See also \cite{BFM}.

The following is a key argument to obtain the finiteness of the abelianization $G^{ab}$ in terms of the periodicity of $\Delta(G)$.

\begin{prop} \label{prop.diagonal}
Let $G$ be a group with finitely generated abelianization. Suppose that the diagonal subgroup $\Delta(G)$ is periodic. Then the abelianization $G^{ab}$ is finite.
\end{prop}

\begin{proof}
As $G^{ab}$ is a finitely generated abelian group we have $$G^{ab} = T \times F,$$ where $T$ is the torsion part and $F$ the free part of $G^{ab}$ (cf. \cite[4.2.10]{Rob}). From \cite[Remark 5]{NR2} we conclude that  $\Delta(G^{ab})$ is isomorphic to $$\Delta(T) \times \Delta(F) \times (T \otimes_{\mathbb{Z}} F),$$ where $T \otimes_{\mathbb{Z}} F$ is the usual tensor product of $\mathbb{Z}$-modules. In particular, the free part of $\Delta(G^{ab})$ is precisely $\Delta(F)$. Now, the canonical projection $G \twoheadrightarrow{G^{ab}}$ induces an epimorphism $q: \Delta(G) \twoheadrightarrow {\Delta(G^{ab})}$. Since $\Delta(G)$ is periodic, it follows that $\Delta(G^{ab})$ is also periodic. Consequently, $F$ is trivial and thus $G^{ab}$ is periodic and, consequently, finite. The proof is complete.
\end{proof}

The proof of Theorem B is now easy to carry out.

\begin{proof}[Proof of Theorem B] $(a) \Rightarrow (b).$  By \cite[Theorem 1]{M}, the non-abelian tensor square $[G,G^{\varphi}]$ is locally finite. \\

$(b) \Rightarrow (c).$ Note that $G'$ is isomorphic to the factor group $[G,G^{\varphi}]/\mu(G)$ and $\Delta(G)$ is a subgroup of $[G,G^{\varphi}]$. Since $[G,G^{\varphi}]$ is locally finite, it follows that the derived subgroup $G'$ and $\Delta(G)$ are locally finite. \\

$(c) \Rightarrow (a).$ By Proposition \ref{prop.diagonal}, $G^{ab}$ is finite. Since $G'$ is locally finite, it follows that $G$ is locally finite (Schmidt, \cite[14.3.1]{Rob}). The proof is complete.
\end{proof}

\begin{rem} \label{weak}
Theorem B may be summarized by saying that if $G$ is a group with finitely generated abelianization $G^{ab}$, such that the diagonal subgroup $\Delta(G)$ is periodic and the derived subgroup $G'$ is locally finite, then $G$ is locally finite. In a certain sense, these conditions cannot be weakened.
\begin{itemize}
\item[$(R_1)$] Note that the local finiteness of the subgroup $\Delta(G)$ alone does not implies the finiteness of the group $G$. For instance, in \cite[Theorem 22]{BM} Blyth and Morse proved that the non-abelian tensor square of the infinite dihedral is  $$[D_{\infty}, D_{\infty}^{\varphi}] \cong C_2 \times C_2 \times C_2 \times C_{\infty},$$ where $D_{\infty} = \langle a,b\mid  a^2=1, a^b = a^{-1}\rangle$ and $\Delta(D_{\infty}) \cong C_2 \times C_2 \times C_2$. In particular, $D_{\infty}$ is finitely generated and the subgroup $\Delta(D_{\infty})$ is finite. However, $D_{\infty}$ is not locally finite.
\item[$(R_2)$] It is clear that the local finiteness of the derived subgroup $G'$ alone does not imply the local finiteness of $G$.
\end{itemize}
\end{rem}

Combining Theorem A and Proposition \ref{prop.diagonal}, one obtains

\begin{cor} \label{c1}
Let $G$ be a group with finitely generated abelianization. Then the following assertions are equivalent.
\begin{itemize}
 \item[$(a)$] There exists only finitely many tensors in $[G,G^{\varphi}]$;
 \item[$(b)$] The non-abelian tensor square $[G,G^{\varphi}]$ is finite;
 \item[$(c)$] The group $\nu(G)$ is finite.
\end{itemize}
\end{cor}

\begin{proof} It is clear that $(c)$ implies $(a)$.

$(a) \Rightarrow (b).$ By Theorem A, the non-abelian tensor square $[G,G^{\varphi}]$ is finite.

$(b) \Rightarrow (c).$ Since $\nu(G) = ([G,G^{\varphi}]\cdot G) \cdot G^{\varphi}$, it is sufficient to prove that $G$ is finite. Since the non-abelian tensor square $[G,G^{\varphi}]$ is finite, it follows that the derived subgroup $G' \cong [G,G^{\varphi}]/\mu(G)$ and the diagonal subgroup $\Delta(G) \leqslant [G,G^{\varphi}]$ are finite. Proposition \ref{prop.diagonal} now shows that the abelianization $G^{ab}$ is finite. The proof is complete.
\end{proof}

Note that Corollary \ref{c1} no longer holds if $G$ is not assumed to be finitely generated.

\begin{rem} \label{prufer}
It is well know that the finiteness of the non-abelian tensor square $G \otimes G$, does not imply that $G$ is a finite group (and so, the group $\nu(G)$ cannot be finite). For instance, the Pr\"ufer group $C_{p^{\infty}}$ is an example of an infinite group such that $T_{\otimes}(C_{p^{\infty}}) = \{0\} = [C_{p^{\infty}}, (C_{p^{\infty}})^{\varphi}]$ and $\nu(C_{p^{\infty}}) = C_{p^{\infty}} \times C_{p^{\infty}}$. Actually,  this is the case for all torsion, divisible abelian groups.
\end{rem}

As usual, if $\pi$ is a nonempty set of primes, a $\pi$-number is a positive integer whose prime divisors belong to $\pi$. An element of a group is called a $\pi$-element if its order is a $\pi$-number. A periodic group $G$ is called a $\pi$-group if every element $g \in G$ is a $\pi$-element. For a periodic group $G$ we denote by $\pi(G)$ the set of all prime divisors of the orders of its elements. If a periodic group $G$ has $\pi(G)=\pi$, then we say that $G$ is a $\pi$-group. If $\pi = \{p\}$ for some prime $p$, it is customary to write {\it $p$-group} rather then {\it $\{p\}$-group}.

\begin{proof}[Proof of Theorem C] Recall that $G$ is a group with finitely generated abelianization and the non-abelian tensor square $[G,G^{\varphi}]$ is a $\pi$-group. We need to show that $G$ is also a $\pi$-group.

Since $G'$ is isomorphic to $[G,G^{\varphi}]/\mu(G)$, we deduce that $G'$ is a $\pi$-group. Now, we only need to show that $G^{ab}$ is a $\pi$-group. By Proposition \ref{prop.diagonal}, the abelianization $G^{ab}$ is finite. Suppose that $G^{ab} = C_{n_1} \times \cdots \times C_{n_r}$, where $C_{n_i}$ denotes the cyclic group of order $n_i$. According to \cite[Remark 5]{NR1}, $$\Delta(G^{ab}) \cong \displaystyle \prod_{i=1}^{r} C_{n_i} \times \displaystyle \prod_{j < k} C_{n_{j,k}},$$ where $n_{j,k} = \gcd(n_j,n_k)$. Consequently, $\pi(G^{ab}) = \pi(\Delta(G^{ab}))$. Now, the canonical projection $G \twoheadrightarrow{G^{ab}}$ induces an epimorphism $q: \Delta(G) \twoheadrightarrow {\Delta(G^{ab})}$. In particular, $\pi(\Delta(G^{ab})) \subseteq \pi(\Delta(G)) \subseteq \pi$. The result follows.
\end{proof}

\section{Applications}

A celebrated result due to E.\,I. Zel'manov \cite{ze1,ze2} refers to the positive solution of the {\it Restricted Burnside Problem}: every residually finite group of bounded exponent is locally finite. The methods used in the solution have shown very effective to treat other questions in group theory (see for instance \cite{BRMfM,BS,shu,shu2,w,wize}). Propositions \ref{prop.1} and \ref{prop.2} are in a certain way corollaries of results found in the above references.

\begin{prop} \label{prop.1}
Let $G$ be a residually finite group with finitely generated abelianization. Suppose that the non-abelian tensor square $[G,G^{\varphi}]$ has bounded exponent. Then $G$ is locally finite.
\end{prop}

\begin{proof} It is clear that the derived subgroup $G' \cong [G,G^{\varphi}]/\mu(G)$ is residually finite and has finite exponent. According to Zel'manov's result \cite{ze1,ze2}, the derived subgroup $G'$ is locally finite. By Schmidt's result \cite[14.3.1]{Rob}, it is sufficient to prove that $G^{ab}$ is (locally) finite. Note that the diagonal subgroup $\Delta(G)$ is an abelian group of bounded exponent. In particular, $\Delta(G)$ is periodic. Proposition \ref{prop.diagonal} now shows that $G^{ab}$ is finite, as well. The proof is complete.
\end{proof}

The next result is an immediate consequence of \cite[Theorem A]{BRMfM}.

\begin{lem} \label{lem.BR}
Let $p$ be a prime and $G$ a residually finite group satisfying a non-trivial identity. Suppose that for every tensor $\alpha$ there exists a $p$-power $q = q(\alpha)$ such that $\alpha^q=1$. Then $[G,G^{\varphi}]$ is locally finite.
\end{lem}

\begin{prop} \label{prop.2}
Let $p$ be a prime and $G$ a residually finite group satisfying a non-trivial identity. Suppose that the abelianization $G^{ab}$ is finitely generated and for every tensor $\alpha$ there exists a $p$-power $q=q(\alpha)$ such that $\alpha^q=1$. Then $G$ is locally finite.
\end{prop}

\begin{proof} By Lemma \ref{lem.BR}, the non-abelian tensor square $[G,G^{\varphi}]$ is locally finite. Now, Theorem B shows that $G$ is locally finite.
\end{proof}

Recall that a group is locally residually finite group if every finitely generated subgroup is residually finite. Interesting classes of groups (for instance, residually finite groups, linear groups, locally finite groups) are locally residually finite.

\begin{prop} \label{prop.locally}
Let $G$ be a locally residually finite group. Suppose that the set of tensors $T_{\otimes}(H)$ is finite in $\nu(H)$ for every proper finitely generated subgroup $H$ of $G$. Then $\nu(G)$ is locally finite.
\end{prop}

\begin{proof}
It will be convenient first to prove the theorem under the additional hypothesis that $G$ is finitely generated. By definition, $G$ is residually finite. It follows that $G$ contains a proper subgroup $H$ of finite index. Consequently, $H$ is also finitely generated \cite[1.6.11]{Rob}. Note that the set of tensors $T_{\otimes}(H)$ is finite. Applying Theorem A and \cite[Theorem 3.1]{NP} to $H$, we obtain that $H$ is finite. Hence $G$ is finite, too. In particular, $[G,G^{\varphi}]$ is finite (Brown and Loday, \cite{BL}). Since $\nu(G) = ([G,G^{\varphi}] \cdot G ) \cdot G^{\varphi}$, it follows that $\nu(G)$ is (locally) finite.

Now we drop the assumption that $G$ is finitely generated. Hence, by the previous paragraph, every proper finitely generated subgroup of $G$ is finite. Consequently, $G$ is locally finite. Thus, \cite[Corollary 5]{M} implies that the group $\nu(G)$ is locally finite, as well. The proof is complete.
\end{proof}

In the above theorem the locally residually finite hypothesis is essential. In fact, an important example in the context of periodic groups, due to A. Olshanskii, shows that for every sufficiently large prime $p$ ($p \geqslant 10^{75}$) there exists an infinite simple group $G$ in which every proper subgroup has order $p$ (see \cite{O} for more details). In particular, for every proper finitely generated subgroup $H$ the set of tensors $T_{\otimes}(H)$ is finite and $\nu(G)$ is not locally finite. Moreover, in Proposition \ref{prop.locally}, it is assumed that the set of tensors $T_{\otimes}(H)$ is finite in $\nu(H)$ for all proper finitely generated subgroup $H$. This condition seems very restrictive. But, in a certain way this restriction cannot be weakened. In \cite{G}, Golod proved that for every prime $p$ and a positive integer $d \geqslant 2$ there exists an infinite $d$-generated residually finite group in which every subgroup $H$ generated by at most $(d-1)$ elements is a finite $p$-group. It follows that the non-abelian tensor square $[H,H^{\varphi}]$ is  finite. In particular, the set of tensors $T_{\otimes}(H)$ is finite in $\nu(H)$ for every subgroup $H$  generated by at most $(d-1)$ elements; however, $\nu(G)$ cannot be locally finite.

\end{document}